\documentclass{amsart}%
\usepackage{graphicx}
\usepackage{amsfonts}
\usepackage{amssymb,latexsym}
\usepackage{amsmath}
\usepackage{amsthm}
\usepackage{url}
\usepackage{subfigure}
\usepackage{amssymb}%
\providecommand{\U}[1]{\protect\rule{.1in}{.1in}}
\theoremstyle{plain}
\newtheorem{theorem}{Theorem}
\newtheorem{lemma}[theorem]{Lemma}
\newtheorem{proposition}[theorem]{Proposition}

\newtheorem{problem}{Problem}

\theoremstyle{definition}

\theoremstyle{remark}

\newtheorem{remark}[theorem]{Remark}

\begin{document}
\title{Pseudo-polynomial Functions over Finite Distributive Lattices}
\author{Miguel Couceiro}
\address[Miguel Couceiro]{Mathematics Research Unit, FSTC, University of Luxembourg \\
6, rue Coudenhove-Kalergi, L-1359 Luxembourg, Luxembourg}
\email{miguel.couceiro[at]uni.lu }
\author{Tam\'as Waldhauser}
\address[Tam\'as Waldhauser]{Mathematics Research Unit, FSTC, University of Luxembourg \\
6, rue Coudenhove-Kalergi, L-1359 Luxembourg, Luxembourg, and Bolyai
Institute, University of Szeged, Aradi v\'{e}rtan\'{u}k tere 1, H-6720 Szeged, Hungary}
\email{twaldha@math.u-szeged.hu}
\maketitle

\begin{abstract}
In this paper we extend the authors' previous works \cite{CW1,CW2} by
considering an aggregation model $f\colon X_{1}\times\cdots\times
X_{n}\rightarrow Y$ for arbitrary sets $X_{1},\ldots,X_{n}$ and a finite
distributive lattice $Y$, factorizable as
\[
~f(x_{1},\ldots,x_{n})=p(\varphi_{1}(x_{1}),\ldots,\varphi_{n}(x_{n})),
\]
where $p$ is an $n$-variable lattice polynomial function over $Y$, and each
$\varphi_{k}$ is a map from $X_{k}$ to $Y$. Following the terminology of
\cite{CW1,CW2}, these are referred to as pseudo-polynomial functions.

We present an axiomatization for this class of pseudo-polynomial functions
which differs from the previous ones both in flavour and nature, and develop
general tools which are then used to obtain all possible such factorizations
of a given pseudo-polynomial function.

\end{abstract}

\section{Introduction and Motivation\label{sect intro}}

The Sugeno integral (introduced by Sugeno~\cite{Sug74,Sug77}) remains as one
of the most noteworthy aggregation functions, and this is partially due to the
fact that it provides a meaningful way to fuse or merge values within
universes where essentially no structure, other than an order, is assumed.
Even though primarily defined over real intervals, the concept of Sugeno
integral can be extended to wider domains, namely, distributive lattices, via
the notion of lattice polynomial function (i.e., a combination of variables
and constants using the lattice operations $\wedge$ and $\vee$). As it turned
out, idempotent lattice polynomial functions coincide with (discrete) Sugeno
integrals (see e.g. \cite{CouMar1,Mar09}).

Recently, the Sugeno integral has been generalized via the notion of
quasi-polynomial function (see \cite{CouMar3}) originally defined as a mapping
$f\colon X^{n}\rightarrow X$~ on a bounded chain $X$ and which can be
factorized as
\begin{equation}
~f(x_{1},\ldots,x_{n})=p(\varphi(x_{1}),\ldots,\varphi(x_{n})),
\label{eq:quasi1}%
\end{equation}
where $p\colon X^{n}\rightarrow X$ is a polynomial function and $\varphi\colon
X\rightarrow X$ is an order-preserving map. This notion was later extended in
two ways.

In \cite{CouMar5}, the input and output universes were allowed to be
arbitrary, possibly different, bounded distributive lattices $X$ and $Y$ so
that~ $f\colon X^{n}\rightarrow Y$ is factorizable as in (\ref{eq:quasi1}),
where now $p\colon Y^{n}\rightarrow Y$ and $\varphi\colon X\rightarrow Y$.
These functions appear naturally within the scope of decision making under
uncertainty since they subsume overall preference functionals associated with
Sugeno integrals whose variables are transformed by the utility function
$\varphi$. Several axiomatizations for this function class were proposed, as
well as all possible factorizations described.

In \cite{CW1} and \cite{CW2} a different extension was considered, now
appearing within the realm of multicriteria decision making. Essentially, the
aggregation model was based on functions $f\colon X_{1}\times\cdots\times
X_{n}\rightarrow Y$ for bounded chains $X_{1},\ldots,X_{n}$ and $Y$, which can
be factorized as compositions
\begin{equation}
~f(x_{1},\ldots,x_{n})=p(\varphi_{1}(x_{1}),\ldots,\varphi_{n}(x_{n})),
\label{eq:pseudo1}%
\end{equation}
where $p\colon Y^{n}\rightarrow Y$ is a Sugeno integral, and each $\varphi
_{k}\colon X_{k}\rightarrow Y$ is an order-preserving map. Such functions were
referred to as Sugeno utility functions in \cite{CW1}. Pseudo-polynomial
functions were defined as functions of the form (\ref{eq:pseudo1}), where $p$
is an arbitrary (possibly non-idempotent) lattice polynomial function, and
each $\varphi_{k}$ satisfies a certain boundary condition (which is weaker
than order-preservation). Note that every quasi-polynomial function
(\ref{eq:quasi1}) can be regarded as a pseudo-polynomial function, where
$X_{1}=\cdots= X_{n}=X$ and $\varphi_{1}=\cdots\varphi_{n}=\varphi$. Moreover,
pseudo-polynomial functions naturally subsume Sugeno utility functions,
and~several axiomatizations were established for this function class in
\cite{CW1}. The question of factorizing a given Sugeno utility function into a
composition (\ref{eq:pseudo1}) was addressed in \cite{CW2}, where a method for
producing such a factorization was presented.

In the current paper we extend the previous results by letting $X_{1}%
,\ldots,X_{n}$ to be arbitrary sets and $Y$ to be an arbitrary finite
distributive lattice, thus subsuming the frameworks in \cite{CouMar5,CW1,CW2}.
Moreover, we develop general tools which allow us to produce all possible
factorizations of a given pseudo-polynomial function into compositions
(\ref{eq:pseudo1}) of a lattice polynomial function $p\colon Y^{n}\rightarrow
Y$ with maps $\varphi_{k}\colon X_{k}\rightarrow Y$.

The structure of the paper is as follows. In Section~\ref{sect preliminaries}
we introduce the basic notions and terminology needed throughout the paper,
and recall some preliminary results. For further background on aggregation
functions and their use in decision making, we refer the reader to
\cite{BouDubPraPir09,GraMarMesPap09}; for basics in the theory of lattices,
see \cite{DavPri,Grae03}. In Section~\ref{sect characterization} we develop a
general framework used to derive an axiomatization of pseudo-polynomial
functions of somewhat different nature than those proposed in
~\cite{CouMar5,CW1,CW2}, and which will provide tools for determining all
possible factorizations of given pseudo-polynomial functions in
Section~\ref{sect factorization}. These results are then illustrated in
Section~\ref{sect example} by means of a concrete example, and in
Section~\ref{sect chains} we show how this new procedure can be applied to
derive the algorithm provided in \cite{CW2,CW3}.

\section{Preliminaries\label{sect preliminaries}}

Throughout this paper, $Y$ is assumed to be a finite distributive lattice with
meet and join operations denoted by $\wedge$ and $\vee$, respectively. Being
finite, $Y$ has a least element and a greatest element, denoted by $0$ and
$1$, respectively. By Birkhoff's Representation Theorem \cite{Bir}, $Y$ can be
embedded into $\mathcal{P}\left(  U\right)  $, the power set of a finite set
$U$. Identifying $Y$ with its image under this embedding, we will consider $Y$
as being a sublattice of $\mathcal{P}\left(  U\right)  $ with $0=\emptyset$
and $1=U$. The complement of a set $S\in\mathcal{P}\left(  U\right)  $ will be
denoted by $\overline{S}$. Since $Y$ is closed under intersections, it induces
a closure operator $\operatorname{cl}$ on $U$, and since $Y$ is closed under
unions, it also induces a dual closure operator $\operatorname{int}$ (also
known as \textquotedblleft interior operator\textquotedblright):%
\[
\operatorname{cl}\left(  S\right)  :=\bigwedge_{\substack{y\in Y\\y\geq
S}}y,\quad\operatorname{int}\left(  S\right)  :=\bigvee_{\substack{y\in
Y\\y\leq S}}y.
\]
It is easy to verify that these two operators satisfy the following identities
for any $S_{1},S_{2}\in\mathcal{P}\left(  U\right)  $:
\[
\operatorname{cl}\left(  S_{1}\vee S_{2}\right)  =\operatorname{cl}\left(
S_{1}\right)  \vee\operatorname{cl}\left(  S_{2}\right)  ,\quad
\operatorname{int}\left(  S_{1}\wedge S_{2}\right)  =\operatorname{int}\left(
S_{1}\right)  \wedge\operatorname{int}\left(  S_{2}\right)  .
\]

A function $p\colon Y^{n}\rightarrow Y$ is a \emph{polynomial function} if it
can be obtained as a composition of the lattice operations $\wedge$ and $\vee$
with variables and constants. As observed in \cite{Mar09}, (\emph{discrete})
\emph{Sugeno integrals} coincide exactly with those lattice polynomial
functions $p$ which are idempotent, i.e., satisfy the identity $p\left(
y,\ldots,y\right)  =y$. An important lattice polynomial function (in fact, a
Sugeno integral) is the \emph{median} function $\operatorname{med}\colon
Y^{3}\rightarrow Y$ defined by%
\begin{align*}
\operatorname{med}\left(  y_{1},y_{2},y_{3}\right)   &  =\left(  y_{1}\wedge
y_{2}\right)  \vee\left(  y_{2}\wedge y_{3}\right)  \vee\left(  y_{3}\wedge
y_{1}\right) \\
&  =\left(  y_{1}\vee y_{2}\right)  \wedge\left(  y_{2}\vee y_{3}\right)
\wedge\left(  y_{3}\vee y_{1}\right)  .
\end{align*}
It is useful to observe that the above expressions for the median can be
simplified when two of the arguments are comparable:%
\begin{equation}
\operatorname{med}\left(  s,y,t\right)  =s\vee\left(  t\wedge y\right)  \text{
whenever }s\leq t. \label{eq med(a,x,b) with a<b}%
\end{equation}

Polynomial functions over bounded distributive lattices have very neat
representations, for instance, in disjunctive normal form \cite{Goo67}. To
describe this disjunctive normal form, let us define $\mathbf{1}_{I}$ to be
the \emph{characteristic vector} of $I\subseteq\lbrack n]:=\left\{
1,\ldots,n\right\}  $, i.e., the $n$-tuple in $Y^{n}$ whose $i$-th component
is $1$ if $i\in I$, and $0$ otherwise.

\begin{theorem}
[Goodstein \cite{Goo67}]\label{prop:DNF(f)} A function $p\colon Y^{n}%
\rightarrow Y$ is a polynomial function if and only if
\begin{equation}
p(y_{1},\ldots,y_{n})=\bigvee_{I\subseteq\lbrack n]}\big(p(\mathbf{1}%
_{I})\wedge\bigwedge_{i\in I}y_{i}\big). \label{eq:Good}%
\end{equation}
Furthermore, the function given by \textup{(\ref{eq:Good})} is a Sugeno
integral if and only if $p(\mathbf{0})=0$ and $p(\mathbf{1})=1$.
\end{theorem}

\begin{remark}
\label{remark unary polynomials}Let us note that in the case $n=1$,
Goodstein's theorem shows that unary polynomial functions $p$ are exactly the
functions of the form $p\left(  y\right)  =s\vee\left(  t\wedge y\right)  $
with $s=p\left(  0\right)  \leq p\left(  1\right)  =t$, and these can be
written as $p\left(  y\right)  =\operatorname{med}\left(  s,y,t\right)  $
according to (\ref{eq med(a,x,b) with a<b}).
\end{remark}

Let $X_{1},\ldots,X_{n}$ be arbitrary sets with at least two elements, and for
each $k\in\left[  n\right]  $ let us fix two distinct elements $0_{X_{k}%
},1_{X_{k}}$ of $X_{k}$ . We shall say that a mapping $\varphi_{k}\colon
X_{k}\rightarrow Y$ satisfies the \emph{boundary condition} (for $0_{X_{k}}$
and $1_{X_{k}}$) if for every $x_{k}\in X_{k}$,
\begin{equation}
\varphi_{k}(0_{X_{k}})\leq\varphi_{k}(x_{k})\leq\varphi_{k}(1_{X_{k}}).
\label{eq BC1}%
\end{equation}
Observe that if $X_{k}$ is a partially ordered set with least element
$0_{X_{k}}$ and greatest element $1_{X_{k}}$, and if $\varphi_{k}$ is
order-preserving, then it satisfies the boundary condition (cf. also
Remark~\ref{remark BC1 for chains}). With no danger of ambiguity, we simply
write $0$ and $1$ instead of $0_{X_{k}}$ and $1_{X_{k}}$ in the sequel.

A function $f\colon\prod_{i\in\lbrack n]}X_{i}\rightarrow Y$ is said to be a
\emph{pseudo-polynomial function}, if there is a polynomial function $p\colon
Y^{n}\rightarrow Y$ and there are unary functions $\varphi_{k}\colon
X_{k}\rightarrow Y\left(  k\in\lbrack n]\right)  $, satisfying the boundary
condition, such that
\begin{equation}
f(\mathbf{x})=p\left(  \boldsymbol{\varphi}\left(  \mathbf{x}\right)  \right)
=p(\varphi_{1}(x_{1}),\ldots,\varphi_{n}(x_{n})) \label{eq:generalPol}%
\end{equation}
holds for all $\mathbf{x}=\left(  x_{1},\ldots,x_{n}\right)  \in\prod
_{i\in\lbrack n]}X_{i}$. If $p$ is a Sugeno integral, then we say that $f$ is
a \emph{pseudo-Sugeno integral}. As it turns out, the notions of
pseudo-polynomial function and pseudo-Sugeno integral are equivalent. This
result was proved in \cite{CW1,CW3} for chains $Y$, but the proof given there
actually just uses the fact that $Y$ is a distributive lattice, hence it
applies verbatim to our setting.

\begin{proposition}
\label{prop:PseudoPol-Sugeno} A function $f\colon\prod_{i\in\lbrack n]}%
X_{i}\rightarrow Y$ is a pseudo-polynomial function if and only if it is a
pseudo-Sugeno integral.
\end{proposition}

Clearly, if $f$ is a pseudo-polynomial function, then it satisfies the
following $n$-variable analogue of the boundary condition (\ref{eq BC1}):%
\begin{equation}
f\left(  \mathbf{x}_{k}^{0}\right)  \leq f\left(  \mathbf{x}\right)  \leq
f\left(  \mathbf{x}_{k}^{1}\right)  \text{ for all }k\in\left[  n\right]
,\mathbf{x\in}\prod_{i\in\lbrack n]}X_{i}, \label{eq BC}%
\end{equation}
where $\mathbf{x}_{k}^{a}\in\prod_{i\in\lbrack n]}X_{i}$ denotes the $n$-tuple
which coincides with $\mathbf{x}$ in all but the $k$-th component, whose value
is $a$.

\begin{remark}
Note that the particular orderings $\varphi_{k}(0_{X_{k}})\leq\varphi
_{k}(1_{X_{k}})$ and $f\left(  \mathbf{x}_{k}^{0}\right)  \leq f\left(
\mathbf{x}_{k}^{1}\right)  $ in (\ref{eq BC1}) and (\ref{eq BC}) could be
reversed as the choice of $0_{X_{k}}$ and $1_{X_{k}}$ is arbitrary. Hence, the
current notion of boundary condition is not more restrictive than the one used
in \cite{CW3}.
\end{remark}

Next we define a property that can be used to characterize pseudo-polynomial
functions. We say that $f\colon\prod_{i\in\lbrack n]}X_{i}\rightarrow Y$ is
\emph{pseudo-median decomposable} if for each $k\in\lbrack n]$ there is a
unary function $\varphi_{k}\colon X_{k}\rightarrow Y$ satisfying
(\ref{eq BC1}), such that
\begin{equation}
f(\mathbf{x})=\operatorname{med}\big(f(\mathbf{x}_{k}^{0}),\varphi_{k}%
(x_{k}),f(\mathbf{x}_{k}^{1})\big) \label{eq pseudo-median decomposition}%
\end{equation}
for every $\mathbf{x}\in\prod_{i\in\lbrack n]}X_{i}$. Note that if $f$ is
pseudo-median decomposable w.r.t. unary functions $\varphi_{k}\colon
X_{k}\rightarrow Y\left(  k\in\lbrack n]\right)  $ satisfying (\ref{eq BC1}),
then (\ref{eq BC}) holds.

The following theorem shows that every pseudo-median decomposable function is
a pseudo-polynomial function, and provides a disjunctive normal form of a
polynomial function $p_{0}$ which can be used to factorize $f$. This theorem
appears in \cite{CW2,CW3} for the special case of chains; for the sake of
self-containedness, we reproduce the proof here. We use the notation
$\widehat{\mathbf{1}}_{I}$ for the characteristic vector of $I\subseteq\lbrack
n]$ in $\prod_{i\in\lbrack n]}X_{i}$, i.e., $\widehat{\mathbf{1}}_{I}\in
\prod_{i\in\lbrack n]}X_{i}$ is the $n$-tuple whose $i$-th component is
$1_{X_{i}}$ if $i\in I$, and $0_{X_{i}}$ otherwise.

\begin{theorem}
\label{thm dnf for pseudo-Sugeno}If $f\colon\prod_{i\in\lbrack n]}%
X_{i}\rightarrow Y$ is pseudo-median decomposable w.r.t. unary functions
$\varphi_{k}\colon X_{k}\rightarrow Y\left(  k\in\lbrack n]\right)  $, then
$f\left(  \mathbf{x}\right)  =p_{0}(\boldsymbol{\varphi}\left(  \mathbf{x}%
\right)  )$, where the polynomial function $p_{0}$ is given by%
\begin{equation}
p_{0}\left(  y_{1},\ldots,y_{n}\right)  =\bigvee\limits_{I\subseteq\left[
n\right]  }\bigl(f\bigl(\widehat{\mathbf{1}}_{I}\bigr)\wedge\bigwedge
\limits_{i\in I}y_{i}\bigr). \label{eq p0 dnf}%
\end{equation}

\end{theorem}

\begin{proof}
We need to prove that the following identity holds:
\begin{equation}
f\left(  x_{1},\ldots,x_{n}\right)  =\bigvee\limits_{I\subseteq\left[
n\right]  }\bigl(f\bigl(\widehat{\mathbf{1}}_{I}\bigr)\wedge\bigwedge
\limits_{i\in I}\varphi_{i}\left(  x_{i}\right)  \bigr).
\label{eq dnf for overall utility}%
\end{equation}
We proceed by induction on $n$. If $n=1$, then the right hand side of
(\ref{eq dnf for overall utility}) takes the form $f\left(  0\right)
\vee\left(  f\left(  1\right)  \wedge\varphi_{1}\left(  x_{1}\right)  \right)
$. From (\ref{eq BC}) it follows that $f\left(  0\right)  \leq f\left(
1\right)  $, and then, using (\ref{eq med(a,x,b) with a<b}), we can rewrite
$f\left(  0\right)  \vee\left(  f\left(  1\right)  \wedge\varphi_{1}\left(
x_{1}\right)  \right)  $ as $\operatorname{med}\left(  f\left(  0\right)
,\varphi_{1}\left(  x_{1}\right)  ,f\left(  1\right)  \right)  $, which equals
$f\left(  x_{1}\right)  $ by (\ref{eq pseudo-median decomposition}).

Now suppose that the statement of the theorem is true for all pseudo-median
decomposable functions in $n-1$ variables. Let $f_{0}$ and $f_{1}$ be the
$(n-1)$-ary functions defined by%
\begin{align*}
f_{0}\left(  x_{1},\ldots,x_{n-1}\right)   &  =f\left(  x_{1},\ldots
,x_{n-1},0\right)  ,\\
f_{1}\left(  x_{1},\ldots,x_{n-1}\right)   &  =f\left(  x_{1},\ldots
,x_{n-1},1\right)  .
\end{align*}
Observe that (\ref{eq BC}) implies $f_{0}\leq f_{1}$. Applying the
pseudo-median decomposition to $f$ with $k=n$ and rewriting the median using
(\ref{eq med(a,x,b) with a<b}), we obtain
\begin{align}
f\left(  x_{1},\ldots,x_{n}\right)   &  =\operatorname{med}\left(
f_{0}\left(  x_{1},\ldots,x_{n-1}\right)  ,\varphi_{n}\left(  x_{n}\right)
,f_{1}\left(  x_{1},\ldots,x_{n-1}\right)  \right)
\label{eq med in nth variable}\\
&  =f_{0}\left(  x_{1},\ldots,x_{n-1}\right)  \vee\left(  f_{1}\left(
x_{1},\ldots,x_{n-1}\right)  \wedge\varphi_{n}\left(  x_{n}\right)  \right)
.\nonumber
\end{align}
It is easy to verify that $f_{0}$ and $f_{1}$ are pseudo-median decomposable
w.r.t. $\varphi_{1},\ldots,\varphi_{n-1}$, therefore we can apply the
induction hypothesis to these functions:%
\begin{align*}
f_{0}\left(  x_{1},\ldots,x_{n-1}\right)  =  &  \bigvee\limits_{I\subseteq
\left[  n-1\right]  }\bigl(f_{0}\bigl(\widehat{\mathbf{1}}_{I}\bigr)\wedge
\bigwedge\limits_{i\in I}\varphi_{i}\left(  x_{i}\right)  \bigr)=\\
&  \bigvee\limits_{I\subseteq\left[  n-1\right]  }\bigl(f\bigl(\widehat
{\mathbf{1}}_{I}\bigr)\wedge\bigwedge\limits_{i\in I}\varphi_{i}\left(
x_{i}\right)  \bigr),\\
f_{1}\left(  x_{1},\ldots,x_{n-1}\right)  =  &  \bigvee\limits_{I\subseteq
\left[  n-1\right]  }\bigl(f_{1}\bigl(\widehat{\mathbf{1}}_{I}\bigr)\wedge
\bigwedge\limits_{i\in I}\varphi_{i}\left(  x_{i}\right)  \bigr)=\\
&  \bigvee\limits_{I\subseteq\left[  n-1\right]  }\bigl(f\bigl(\widehat
{\mathbf{1}}_{I\cup\left\{  n\right\}  }\bigr)\wedge\bigwedge\limits_{i\in
I}\varphi_{i}\left(  x_{i}\right)  \bigr).
\end{align*}
Substituting back into (\ref{eq med in nth variable}) and using distributivity
we obtain the desired equality (\ref{eq dnf for overall utility}).
\end{proof}

Now we can prove that pseudo-median decomposability actually characterizes
pseudo-polynomial functions (see \cite{CW1,CW3} for the case of chains, where
the proof is slightly simpler).

\begin{theorem}
\label{thm pseudo-median decomposition} Let $f\colon\prod_{i\in\lbrack
n]}X_{i}\rightarrow Y$ be a function. Then $f$ is a pseudo-polynomial function
if and only if $f$ is pseudo-median decomposable.
\end{theorem}

\begin{proof}
Sufficiency follows from Theorem~\ref{thm dnf for pseudo-Sugeno}, so we only
need to show that if $f$ is a pseudo-polynomial function, then it is
pseudo-median decomposable. Suppose that $f(\mathbf{x})=p\left(
\boldsymbol{\varphi}\left(  \mathbf{x}\right)  \right)  $ as in
(\ref{eq:generalPol}), and let $k\in\left[  n\right]  $. We have to prove that
(\ref{eq pseudo-median decomposition}) holds for all $\mathbf{x}\in\prod
_{i\in\lbrack n]}X_{i}$. Regarding $x_{i}$ as a fixed element of $X_{i}$ for
each $i\neq k$, we can define a unary polynomial function $u\colon
Y\rightarrow Y$ by%
\[
u\left(  y\right)  =p(\varphi_{1}(x_{1}),\ldots,\varphi_{k-1}(x_{k-1}%
),y,\varphi_{k+1}(x_{k+1}),\ldots,\varphi_{n}(x_{n})).
\]

To simplify notation, let us write $a:=\varphi_{k}\left(  0\right)
,z:=\varphi_{k}\left(  x_{k}\right)  ,b:=\varphi_{k}\left(  1\right)  $, and
let us note that the boundary condition (\ref{eq BC}) yields $a\leq z\leq b$.
With this notation (\ref{eq pseudo-median decomposition}) reads as $u\left(
z\right)  =\operatorname{med}\left(  u\left(  a\right)  ,z,u\left(  b\right)
\right)  $. In order to verify this equality, we write $u$ in disjunctive
normal form as in Remark~\ref{remark unary polynomials}: $u\left(  y\right)
=s\vee\left(  t\wedge y\right)  $, where $s\leq t$. Now the proof is a
straightforward computation, making heavy use of distributivity:%

\begin{align*}
\operatorname{med}\left(  u\left(  a\right)  ,z,u\left(  b\right)  \right)
&  =u\left(  a\right)  \vee\left(  u\left(  b\right)  \wedge z\right) \\
&  =\left(  s\vee\left(  t\wedge a\right)  \right)  \vee\left(  \left(
s\vee\left(  t\wedge b\right)  \right)  \wedge z\right) \\
&  =s\vee\left(  t\wedge a\right)  \vee\left(  s\wedge z\right)  \vee\left(
t\wedge b\wedge z\right) \\
&  =s\vee\left(  t\wedge a\right)  \vee\left(  s\wedge z\right)  \vee\left(
t\wedge z\right) \\
&  =s\vee\left(  t\wedge a\right)  \vee\left(  t\wedge z\right) \\
&  =s\vee\left(  t\wedge\left(  a\vee z\right)  \right) \\
&  =s\vee\left(  t\wedge z\right)  =u\left(  z\right)  .\qedhere
\end{align*}

\end{proof}

\section{Characterization of Pseudo-polynomial Functions
\label{sect characterization}}

Let $f\colon\prod_{i\in\lbrack n]}X_{i}\rightarrow Y$ be a function satisfying
(\ref{eq BC}), and for each $k\in\left[  n\right]  $ let us define two
auxiliary functions $\Phi_{k}^{-},\Phi_{k}^{+}\colon X_{k}\rightarrow Y$ as
follows:%
\begin{equation}
\Phi_{k}^{-}\left(  a_{k}\right)  :=\bigvee_{x_{k}=a_{k}}\operatorname{cl}%
\bigl(f\left(  \mathbf{x}\right)  \wedge\overline{f\left(  \mathbf{x}_{k}%
^{0}\right)  }\bigr)\text{,\quad}\Phi_{k}^{+}\left(  a_{k}\right)
:=\bigwedge_{x_{k}=a_{k}}\operatorname{int}\bigl(f\left(  \mathbf{x}\right)
\vee\overline{f\left(  \mathbf{x}_{k}^{1}\right)  }\bigr).
\label{eq definition of  Fi_k- and Fi_k+}%
\end{equation}
Here the join and the meet range over all $\mathbf{x}\in\prod_{i\in\lbrack
n]}X_{i}$ whose $k$-th component is $a_{k}$. Note that from (\ref{eq BC}) it
follows that $\Phi_{k}^{-}$ and $\Phi_{k}^{+}$ satisfy the boundary condition
(\ref{eq BC1}). With the help of these functions, we will give a necessary and
sufficient condition for $f$ to be a pseudo-polynomial function. The following
lemma formulates a simple observation that allows us to solve equation
(\ref{eq pseudo-median decomposition}) for $\varphi_{k}(x_{k})$.

\begin{lemma}
\label{lemma med(a,x,b)=m}For any $u\leq m\leq w,v\in Y$ the following two
conditions are equivalent:%
\renewcommand{\labelenumi}{\textup{(\roman{enumi})}}
\renewcommand{\theenumi}{\labelenumi}%

\begin{enumerate}
\item \label{item 1 in lemma med(a,x,b)=m}$\operatorname{med}\left(
u,v,w\right)  =m$;

\item \label{item 2 in lemma med(a,x,b)=m}$m\wedge\overline{u}\leq v\leq
m\vee\overline{w}$.
\end{enumerate}
\end{lemma}

\begin{proof}
Assuming that $\operatorname{med}\left(  u,v,w\right)  =m$, we can estimate
$m\wedge\overline{u}$ using (\ref{eq med(a,x,b) with a<b}) as follows:%
\begin{align*}
m\wedge\overline{u}  &  =%
\bigl(%
u\vee\left(  v\wedge w\right)
\bigr)%
\wedge\overline{u}\\
&  =\left(  u\wedge\overline{u}\right)  \vee\left(  v\wedge w\wedge
\overline{u}\right) \\
&  =0\vee\left(  v\wedge w\wedge\overline{u}\right)  \leq v\text{.}%
\end{align*}
An analogous argument shows that $v\leq m\vee\overline{w}$, and this
establishes \ref{item 1 in lemma med(a,x,b)=m}$\implies$%
\ref{item 2 in lemma med(a,x,b)=m}.

In order to prove \ref{item 2 in lemma med(a,x,b)=m}$\implies$%
\ref{item 1 in lemma med(a,x,b)=m}, let us first compute $\operatorname{med}%
\left(  u,m\wedge\overline{u},w\right)  $, again with the help of
(\ref{eq med(a,x,b) with a<b}):%
\begin{align*}
\operatorname{med}\left(  u,m\wedge\overline{u},w\right)   &  =u\vee\left(
m\wedge\overline{u}\wedge w\right) \\
&  =\left(  u\vee m\right)  \wedge\left(  u\vee\overline{u}\right)
\wedge\left(  u\vee w\right) \\
&  =m\wedge1\wedge w=m\text{.}%
\end{align*}
Similarly, we have $\operatorname{med}\left(  u,m\vee\overline{w},w\right)
=m$, and then, using \ref{item 2 in lemma med(a,x,b)=m} and the monotonicity
of the median function, we conclude%
\[
m=\operatorname{med}\left(  u,m\wedge\overline{u},w\right)  \leq
\operatorname{med}\left(  u,v,w\right)  \leq\operatorname{med}\left(
u,m\vee\overline{w},w\right)  =m,
\]
hence $\operatorname{med}\left(  u,v,w\right)  =m$.
\end{proof}

With the help of the above lemma we derive from
Theorem~\ref{thm pseudo-median decomposition} a necessary condition for $f$ to
be a pseudo-polynomial function.

\begin{proposition}
\label{prop necessary}If $f\colon\prod_{i\in\lbrack n]}X_{i}\rightarrow Y$ is
a pseudo-polynomial function, then it satisfies \textup{(\ref{eq BC})} and%
\begin{equation}
\Phi_{k}^{-}\leq\Phi_{k}^{+}\text{,\quad for all }k\in\left[  n\right]  .
\label{eq Fi_k-+}%
\end{equation}

\end{proposition}

\begin{proof}
Let us suppose that $f\left(  \mathbf{x}\right)  =p\left(  \boldsymbol{\varphi
}\left(  \mathbf{x}\right)  \right)  $ is a pseudo-polynomial function. Then
(\ref{eq pseudo-median decomposition}) holds by
Theorem~\ref{thm pseudo-median decomposition}, and applying
Lemma~\ref{lemma med(a,x,b)=m} with $u=f\left(  \mathbf{x}_{k}^{0}\right)
,m=f\left(  \mathbf{x}\right)  ,w=f\left(  \mathbf{x}_{k}^{1}\right)  $ and
$v=\varphi_{k}\left(  x_{k}\right)  $, we see that $f\left(  \mathbf{x}%
\right)  \wedge\overline{f\left(  \mathbf{x}_{k}^{0}\right)  }\leq\varphi
_{k}\left(  x_{k}\right)  \leq f\left(  \mathbf{x}\right)  \vee\overline
{f\left(  \mathbf{x}_{k}^{1}\right)  }$. Moreover, since $\varphi_{k}\left(
x_{k}\right)  \in Y$, we have%
\[
\operatorname{cl}\bigl(f\left(  \mathbf{x}\right)  \wedge\overline{f\left(
\mathbf{x}_{k}^{0}\right)  }\bigr)\leq\varphi_{k}\left(  x_{k}\right)
\leq\operatorname{int}\bigl(f\left(  \mathbf{x}\right)  \vee\overline{f\left(
\mathbf{x}_{k}^{1}\right)  }\bigr).
\]
Considering these inequalities for all $\mathbf{x\in}\prod_{i\in\lbrack
n]}X_{i}$ with a fixed $k$-th component $x_{k}=a_{k}$, it follows that
\begin{equation}
\Phi_{k}^{-}\left(  a_{k}\right)  \leq\varphi_{k}\left(  a_{k}\right)
\leq\Phi_{k}^{+}\left(  a_{k}\right)  \label{eq Fi-<fi<Fi+}%
\end{equation}
for all $k\in\left[  n\right]  ,a_{k}\in X_{k}$.
\end{proof}

\begin{remark}
\label{remark eq Fi_k-+ reformulated}Let us note that (\ref{eq Fi_k-+}) holds
if and only if each joinand in the definition of $\Phi_{k}^{-}\left(
a_{k}\right)  $ is less than or equal to each meetand in the definition of
$\Phi_{k}^{+}\left(  a_{k}\right)  $. In other words, (\ref{eq Fi_k-+}) is
equivalent to%
\[
\operatorname{cl}\bigl(f\left(  \mathbf{y}\right)  \wedge\overline{f\left(
\mathbf{y}_{k}^{0}\right)  }\bigr)\leq\operatorname{int}\bigl(f\left(
\mathbf{x}\right)  \vee\overline{f\left(  \mathbf{x}_{k}^{1}\right)
}\bigr)\text{ for all }\mathbf{x},\mathbf{y}\in\prod_{i\in\lbrack n]}X\text{
with }x_{k}=y_{k}.
\]

\end{remark}

In order to prove that the necessary condition presented in the above
proposition is also sufficient, we verify that (\ref{eq BC}) and
(\ref{eq Fi_k-+}) imply that $f$ is pseudo-median decomposable with respect to
$\Phi_{1}^{-},\ldots,\Phi_{n}^{-}$ and also with respect to $\Phi_{1}%
^{+},\ldots,\Phi_{n}^{+}$.

\begin{proposition}
\label{prop sufficient}Suppose that $f\colon\prod_{i\in\lbrack n]}%
X_{i}\rightarrow Y$ satisfies \textup{(\ref{eq BC})} and
\textup{(\ref{eq Fi_k-+})}. Then, for all $\mathbf{x}\in\prod_{i\in\lbrack
n]}X_{i}$ and $k\in\left[  n\right]  $, we have%
\[
f(\mathbf{x})=\operatorname{med}\big(f(\mathbf{x}_{k}^{0}),\Phi_{k}^{-}%
(x_{k}),f(\mathbf{x}_{k}^{1})\big)=\operatorname{med}\big(f(\mathbf{x}_{k}%
^{0}),\Phi_{k}^{+}(x_{k}),f(\mathbf{x}_{k}^{1})\big).
\]

\end{proposition}

\begin{proof}
We prove that $f$ is pseudo-median decomposable with respect to $\Phi_{1}%
^{-},\ldots,\Phi_{n}^{-}$, and leave to the reader the analogous argument for
$\Phi_{1}^{+},\ldots,\Phi_{n}^{+}$. Let us fix $k\in\left[  n\right]
,a_{k}\in X_{k}$ and $\mathbf{x}\in\prod_{i\in\lbrack n]}X_{i}$ with
$x_{k}=a_{k}$. By the definition of $\Phi_{k}^{-}$, we have%
\[
\operatorname{med}\big(f(\mathbf{x}_{k}^{0}),\Phi_{k}^{-}(a_{k}),f(\mathbf{x}%
_{k}^{1})\big)=\operatorname{med}\Big(f(\mathbf{x}_{k}^{0}),\bigvee
_{y_{k}=a_{k}}\operatorname{cl}\bigl(f\left(  \mathbf{y}\right)
\wedge\overline{f\left(  \mathbf{y}_{k}^{0}\right)  }\bigr),f(\mathbf{x}%
_{k}^{1})\Big).
\]
From the distributivity of $Y$ it follows that joins distribute over medians,
and this gives the following expression for $\operatorname{med}%
\big(f(\mathbf{x}_{k}^{0}),\Phi_{k}^{-}(a_{k}),f(\mathbf{x}_{k}^{1})\big)$:%

\begin{equation}
\operatorname{med}\big(f(\mathbf{x}_{k}^{0}),\Phi_{k}^{-}(a_{k}),f(\mathbf{x}%
_{k}^{1})\big)=\bigvee_{y_{k}=a_{k}}\operatorname{med}%
\bigl(%
f(\mathbf{x}_{k}^{0}),\operatorname{cl}\bigl(f\left(  \mathbf{y}\right)
\wedge\overline{f\left(  \mathbf{y}_{k}^{0}\right)  }\bigr),f(\mathbf{x}%
_{k}^{1})%
\bigr)%
. \label{eq pseudo-median with Fi_k-}%
\end{equation}

We can estimate this join from below by keeping only the joinand corresponding
to $\mathbf{y}=\mathbf{x}$ (this indeed appears in the join, since
$x_{k}=a_{k}$):%
\begin{equation}
\operatorname{med}\big(f(\mathbf{x}_{k}^{0}),\Phi_{k}^{-}(a_{k}),f(\mathbf{x}%
_{k}^{1})\big)\geq\operatorname{med}\Big(f(\mathbf{x}_{k}^{0}%
),\operatorname{cl}\bigl(f\left(  \mathbf{x}\right)  \wedge\overline{f\left(
\mathbf{x}_{k}^{0}\right)  }\bigr),f(\mathbf{x}_{k}^{1})\Big).
\label{eq first lower estimate for pseudo-median with Fi_k-}%
\end{equation}
Applying Lemma~\ref{lemma med(a,x,b)=m} with $u=f(\mathbf{x}_{k}^{0})$,
$v=\operatorname{cl}\bigl(f\left(  \mathbf{x}\right)  \wedge\overline{f\left(
\mathbf{x}_{k}^{0}\right)  }\bigr)$, $w=f(\mathbf{x}_{k}^{1})$ and $m=f\left(
\mathbf{x}\right)  $ and taking into account that $f\left(  \mathbf{x}_{k}%
^{0}\right)  \leq f\left(  \mathbf{x}\right)  \leq f\left(  \mathbf{x}_{k}%
^{1}\right)  $ holds by (\ref{eq BC}), we see that the right hand side of
(\ref{eq first lower estimate for pseudo-median with Fi_k-}) equals $f\left(
\mathbf{x}\right)  $. This yields the inequality%
\begin{equation}
\operatorname{med}\big(f(\mathbf{x}_{k}^{0}),\Phi_{k}^{-}(a_{k}),f(\mathbf{x}%
_{k}^{1})\big)\geq f\left(  \mathbf{x}\right)  .
\label{eq final lower estimate for pseudo-median with Fi_k-}%
\end{equation}

In order to prove the converse inequality, let us note that property
(\ref{eq Fi_k-+}) implies%
\[
\operatorname{cl}\bigl(f\left(  \mathbf{y}\right)  \wedge\overline{f\left(
\mathbf{y}_{k}^{0}\right)  }\bigr)\leq\operatorname{int}\bigl(f\left(
\mathbf{x}\right)  \vee\overline{f\left(  \mathbf{x}_{k}^{1}\right)  }\bigr),
\]
whenever $y_{k}=a_{k}$ (see Remark~\ref{remark eq Fi_k-+ reformulated}). Thus,
replacing $\operatorname{cl}\bigl(f\left(  \mathbf{y}\right)  \wedge
\overline{f\left(  \mathbf{y}_{k}^{0}\right)  }\bigr)$ by $\operatorname{int}%
\bigl(f\left(  \mathbf{x}\right)  \vee\overline{f\left(  \mathbf{x}_{k}%
^{1}\right)  }\bigr)$ in each joinand on the right hand side of
(\ref{eq pseudo-median with Fi_k-}), we get the upper estimate
\begin{equation}
\operatorname{med}\big(f(\mathbf{x}_{k}^{0}),\Phi_{k}^{-}(a_{k}),f(\mathbf{x}%
_{k}^{1})\big)\leq\operatorname{med}%
\bigl(%
f(\mathbf{x}_{k}^{0}),\operatorname{int}\bigl(f\left(  \mathbf{x}\right)
\vee\overline{f\left(  \mathbf{x}_{k}^{1}\right)  }\bigr),f(\mathbf{x}_{k}%
^{1})%
\bigr)%
. \label{eq first upper estimate for pseudo-median with Fi_k-}%
\end{equation}
Again, Lemma~\ref{lemma med(a,x,b)=m} shows that the right hand side of
(\ref{eq first upper estimate for pseudo-median with Fi_k-}) equals $f\left(
\mathbf{x}\right)  $, hence we have%
\begin{equation}
\operatorname{med}\big(f(\mathbf{x}_{k}^{0}),\Phi_{k}^{-}(a_{k}),f(\mathbf{x}%
_{k}^{1})\big)\leq f\left(  \mathbf{x}\right)  .
\label{eq final upper estimate for pseudo-median with Fi_k-}%
\end{equation}

Combining inequalities
(\ref{eq final lower estimate for pseudo-median with Fi_k-}) and
(\ref{eq final upper estimate for pseudo-median with Fi_k-}), we get the
desired equality%
\[
\operatorname{med}\big(f(\mathbf{x}_{k}^{0}),\Phi_{k}^{-}(a_{k}),f(\mathbf{x}%
_{k}^{1})\big)=f\left(  \mathbf{x}\right)  .\qedhere
\]

\end{proof}

Propositions~\ref{prop necessary} and \ref{prop sufficient} together with
Theorem~\ref{thm pseudo-median decomposition} yield the following
characterization of pseudo-polynomial functions.

\begin{theorem}
\label{thm characterization}A function $f\colon\prod_{i\in\lbrack n]}%
X_{i}\rightarrow Y$ is a pseudo-polynomial function if and only if it
satisfies conditions \textup{(\ref{eq BC})} and \textup{(\ref{eq Fi_k-+})}.
\end{theorem}

\begin{remark}
Theorem~\ref{thm characterization} is of different nature than
Theorem~\ref{thm pseudo-median decomposition} and the various
characterizations obtained in \cite{CW3}: here the necessary and sufficient
condition for $f$ being a pseudo-polynomial function is given solely in terms
of $f$ itself, without referring to the existence of certain functions
$\varphi_{k}$.
\end{remark}

\section{Factorizations of Pseudo-polynomial
Functions\label{sect factorization}}

Let us suppose that $f\colon\prod_{i\in\lbrack n]}X_{i}\rightarrow Y$
satisfies (\ref{eq BC}) and (\ref{eq Fi_k-+}). According to
Theorem~\ref{thm characterization}, $f$ is a pseudo-polynomial function, i.e.,
it has a factorization of the form $f\left(  \mathbf{x}\right)  =p\left(
\boldsymbol{\varphi}\left(  \mathbf{x}\right)  \right)  $, where $p\colon
Y^{n}\rightarrow Y$ is a polynomial function and each $\varphi_{k}\colon
X_{k}\rightarrow Y\left(  k\in\left[  n\right]  \right)  $ is a unary map
satisfying (\ref{eq BC1}). We now show how to construct such a factorization;
in fact, we will find all possible factorizations. First we describe the set
of possible functions $\varphi_{k}$ (the polynomial function $p_{0}$ in the
statement of the theorem is the one defined by (\ref{eq p0 dnf}) in
Theorem~\ref{thm dnf for pseudo-Sugeno}).

\begin{theorem}
\label{thm factorization--fi}For any function $f\colon\prod_{i\in\lbrack
n]}X_{i}\rightarrow Y$ satisfying \textup{(\ref{eq BC})} and unary maps
$\varphi_{k}\colon X_{k}\rightarrow Y\left(  k\in\left[  n\right]  \right)  $
satisfying \textup{(\ref{eq BC1})}, the following three conditions are
equivalent:%
\renewcommand{\labelenumi}{\textup{(\roman{enumi})}}
\renewcommand{\theenumi}{\labelenumi}%

\begin{enumerate}
\item \label{item 1 in thm factorization--fi}$\Phi_{k}^{-}\leq\varphi_{k}%
\leq\Phi_{k}^{+}$ holds for all $k\in\left[  n\right]  $;

\item \label{item 2 in thm factorization--fi}$f\left(  \mathbf{x}\right)
=p_{0}\left(  \boldsymbol{\varphi}\left(  \mathbf{x}\right)  \right)  $;

\item \label{item 3 in thm factorization--fi}there exists a polynomial
function $p\colon Y^{n}\rightarrow Y$ such that $f\left(  \mathbf{x}\right)
=p\left(  \boldsymbol{\varphi}\left(  \mathbf{x}\right)  \right)  $.
\end{enumerate}
\end{theorem}

\begin{proof}
The implication \ref{item 2 in thm factorization--fi}$\implies$%
\ref{item 3 in thm factorization--fi} is trivial, and
\ref{item 3 in thm factorization--fi}$\implies$%
\ref{item 1 in thm factorization--fi} has been established in the course of
the proof of Proposition~\ref{prop necessary} (see equation
(\ref{eq Fi-<fi<Fi+})).

So suppose that \ref{item 1 in thm factorization--fi} holds. Then obviously
(\ref{eq Fi_k-+}) holds as well, and Proposition~\ref{prop sufficient} shows
that $f$ is pseudo-median decomposable with respect to $\Phi_{1}^{-}%
,\ldots,\Phi_{n}^{-}$ and also with respect to $\Phi_{1}^{+},\ldots,\Phi
_{n}^{+}$. Since $\Phi_{k}^{-}\leq\varphi_{k}\leq\Phi_{k}^{+}$ holds for all
$k\in\left[  n\right]  $ by \ref{item 1 in thm factorization--fi}, we have
\begin{align*}
f\left(  \mathbf{x}\right)   &  =\operatorname{med}\big(f(\mathbf{x}_{k}%
^{0}),\Phi_{k}^{-}(x_{k}),f(\mathbf{x}_{k}^{1})\big)\\
&  \leq\operatorname{med}\big(f(\mathbf{x}_{k}^{0}),\varphi_{k}(x_{k}%
),f(\mathbf{x}_{k}^{1})\big)\\
&  \leq\operatorname{med}\big(f(\mathbf{x}_{k}^{0}),\Phi_{k}^{+}%
(x_{k}),f(\mathbf{x}_{k}^{1})\big)=f\left(  \mathbf{x}\right)  ,
\end{align*}
therefore $f$ is pseudo-median decomposable with respect to $\varphi
_{1},\ldots,\varphi_{n}$. Now \ref{item 2 in thm factorization--fi} follows
immediately from Theorem~\ref{thm dnf for pseudo-Sugeno}.
\end{proof}

Theorem~\ref{thm factorization--fi} describes all those unary maps
$\varphi_{1},\ldots,\varphi_{n}$ that can occur in a factorization of $f$, but
it does not provide all possible polynomial functions $p$. (We know that
$p_{0}$ can be used in any factorization, but there may be others as well.) To
find all factorizations (\ref{eq:generalPol}) of $f$, let us fix unary
functions $\varphi_{k}\colon X_{k}\rightarrow Y\left(  k\in\left[  n\right]
\right)  $ satisfying (\ref{eq BC1}), such that $\Phi_{k}^{-}\leq\varphi
_{k}\leq\Phi_{k}^{+}$ for each $k\in\left[  n\right]  $. To simplify notation,
let $a_{k}=\varphi_{k}\left(  0_{X_{k}}\right)  ,b_{k}=\varphi_{k}\left(
1_{X_{k}}\right)  $, and for each $I\subseteq\left[  n\right]  $ let
$\mathbf{e}_{I}\in Y^{n}$ be the $n$-tuple whose $i$-th component is $a_{i}$
if $i\notin I$ and $b_{i}$ if $i\in I$. If $p\colon Y^{n}\rightarrow Y$ is a
polynomial function such that $f\left(  \mathbf{x}\right)  =p\left(
\boldsymbol{\varphi}\left(  \mathbf{x}\right)  \right)  $, then%
\begin{equation}
p\left(  \mathbf{e}_{I}\right)  =f\bigl(\widehat{\mathbf{1}}_{I}%
\bigr)\text{\quad for all }I\subseteq\left[  n\right]  \text{,}
\label{eq interpol}%
\end{equation}
since $\mathbf{e}_{I}=\boldsymbol{\varphi}\bigl(\widehat{\mathbf{1}}%
_{I}\bigr)$. We show that (\ref{eq interpol}) is not only necessary but also
sufficient to establish the factorization $f\left(  \mathbf{x}\right)
=p\left(  \boldsymbol{\varphi}\left(  \mathbf{x}\right)  \right)  $.

\begin{lemma}
\label{lemma interpol}Let $f\colon\prod_{i\in\lbrack n]}X_{i}\rightarrow Y$ be
a function satisfying $(\ref{eq BC})$ and $(\ref{eq Fi_k-+})$, and let
$\varphi_{k}\colon X_{k}\rightarrow Y\left(  k\in\left[  n\right]  \right)  $
be maps satisfying $(\ref{eq BC1})$, such that $\Phi_{k}^{-}\leq\varphi
_{k}\leq\Phi_{k}^{+}$ for all $k\in\left[  n\right]  $. Then a polynomial
function $p\colon Y^{n}\rightarrow Y$ yields a factorization $f\left(
\mathbf{x}\right)  =p\left(  \boldsymbol{\varphi}\left(  \mathbf{x}\right)
\right)  $ if and only if $(\ref{eq interpol})$ holds.
\end{lemma}

\begin{proof}
As noted above, necessity is trivial.\ To prove the sufficiency, let us assume
that $p$ satisfies (\ref{eq interpol}), and let us define a function
$f^{\prime}\colon\prod_{i\in\lbrack n]}X_{i}\rightarrow Y$ by $f^{\prime
}\left(  \mathbf{x}\right)  =p\left(  \boldsymbol{\varphi}\left(
\mathbf{x}\right)  \right)  $. Then $f^{\prime}$ is a pseudo-polynomial
function, and by Theorem~\ref{thm pseudo-median decomposition} it is
pseudo-median decomposable with respect to $\varphi_{1},\ldots,\varphi_{n}$.
From Theorem~\ref{thm dnf for pseudo-Sugeno} we get the following expression
for $f^{\prime}$:%
\begin{equation}
f^{\prime}\left(  \mathbf{x}\right)  =\bigvee\limits_{I\subseteq\left[
n\right]  }\bigl(f^{\prime}\bigl(\widehat{\mathbf{1}}_{I}\bigr)\wedge
\bigwedge\limits_{i\in I}\varphi_{i}\left(  x_{i}\right)  \bigr).
\label{eq DNF for f' in lemma interpol}%
\end{equation}

The assumptions on $f$ and $\varphi_{k}$ guarantee that $f\left(
\mathbf{x}\right)  =p_{0}\left(  \boldsymbol{\varphi}\left(  \mathbf{x}%
\right)  \right)  $ by Theorem~\ref{thm factorization--fi}, and hence $f$ can
be written as%
\begin{equation}
f\left(  \mathbf{x}\right)  =\bigvee\limits_{I\subseteq\left[  n\right]
}\bigl(f\bigl(\widehat{\mathbf{1}}_{I}\bigr)\wedge\bigwedge\limits_{i\in
I}\varphi_{i}\left(  x_{i}\right)  \bigr).
\label{eq DNF for f in lemma interpol}%
\end{equation}

Since $f^{\prime}\bigl(\widehat{\mathbf{1}}_{I}%
\bigr)=p\bigl(\boldsymbol{\varphi}\bigl(\widehat{\mathbf{1}}_{I}%
\bigr)\bigr)=p\left(  \mathbf{e}_{I}\right)  $ and $p\left(  \mathbf{e}%
_{I}\right)  =f\bigl(\widehat{\mathbf{1}}_{I}\bigr)$ by (\ref{eq interpol}),
it follows from (\ref{eq DNF for f' in lemma interpol}) and
(\ref{eq DNF for f in lemma interpol}) that $f^{\prime}\left(  \mathbf{x}%
\right)  =f\left(  \mathbf{x}\right)  $.
\end{proof}

For a given $f$ and given $a_{k},b_{k}\in Y$, (\ref{eq interpol}) gives rise
to a polynomial interpolation problem over $Y$: the values of the unknown
polynomial function $p$ are prescribed at certain ($2^{n}$ many) points in
$Y^{n}$. It has been shown in \cite{CWgoodstein} that the least solution of
this interpolation problem is%
\[
p^{-}\left(  \mathbf{y}\right)  =\bigvee_{I\subseteq\left[  n\right]
}\bigl(c_{I}^{-}\wedge\bigwedge_{i\in I}y_{i}\bigr)\text{,\quad where }%
c_{I}^{-}=\operatorname{cl}\bigl(f\bigl(\widehat{\mathbf{1}}_{I}%
\bigr)\wedge\bigwedge_{i\notin I}\overline{a_{i}}\bigr),
\]
whereas the greatest solution is
\[
p^{+}\left(  \mathbf{y}\right)  =\bigvee_{I\subseteq\left[  n\right]
}\bigl(c_{I}^{+}\wedge\bigwedge_{i\in I}y_{i}\bigr)\text{,\quad where }%
c_{I}^{+}=\operatorname{int}\bigl(f\bigl(\widehat{\mathbf{1}}_{I}%
\bigr)\vee\bigvee_{i\in I}\overline{b_{i}}\bigr).
\]
In other words, a polynomial function $p$ is a solution of (\ref{eq interpol})
if and only if $p^{-}\leq p\leq p^{+}$. Since, by Theorem~\ref{prop:DNF(f)},
$p$ is uniquely determined by its values on the tuples $\mathbf{1}_{I}\left(
I\subseteq\left[  n\right]  \right)  $, this is equivalent to%
\[
c_{I}^{-}=p^{-}\left(  \mathbf{1}_{I}\right)  \leq p\left(  \mathbf{1}%
_{I}\right)  \leq p^{+}\left(  \mathbf{1}_{I}\right)  =c_{I}^{+}\text{\quad
for all }I\subseteq\left[  n\right]  .
\]
Thus we obtain the following description of all possible factorizations of a
given pseudo-polynomial function $f$.

\begin{theorem}
\label{thm factrization--fi and p}Let $f\colon\prod_{i\in\lbrack n]}%
X_{i}\rightarrow Y$ be a function satisfying \textup{(\ref{eq BC})}, for each
$k\in\lbrack n]$ let $\varphi_{k}\colon X_{k}\rightarrow Y$ be a given
function satisfying \textup{(\ref{eq BC1})}, and let $p\colon Y^{n}\rightarrow
Y$ be a polynomial function. Then $f\left(  \mathbf{x}\right)  =p\left(
\boldsymbol{\varphi}\left(  \mathbf{x}\right)  \right)  $ if and only if
$\Phi_{k}^{-}\leq\varphi_{k}\leq\Phi_{k}^{+}$ for each $k\in\left[  n\right]
$, and we have $p^{-}\leq p\leq p^{+}$.
\end{theorem}

\begin{remark}
Note that the polynomial functions $p^{-}$ and $p^{+}$ are defined in terms of
the maps $\varphi_{k}$, hence we have to choose these maps first, and only
then can we determine $p^{-}$ and $p^{+}$ (cf. the example in
Section~\ref{sect example}).
\end{remark}

\begin{remark}
Clearly, $c_{I}^{-}\leq f\bigl(\widehat{\mathbf{1}}_{I}\bigr)\leq c_{I}^{+}$
holds independently of $a_{k},b_{k}$, hence the polynomial function $p_{0}$
can be used in any factorization of $f$, as it was already shown in
Theorem~\ref{thm dnf for pseudo-Sugeno}.
\end{remark}

\begin{remark}
If $X_{k}$ is a partially ordered set for each $k\in\left[  n\right]  $ and
$f$ is order-preserving, then $\Phi_{k}^{-}$ and $\Phi_{k}^{+}$ are also
order-preserving. This shows that every order-preserving pseudo-polynomial
function has a factorization where each $\varphi_{k}$ is order-preserving.
Consequently, order-preserving pseudo-Sugeno integrals coincide with Sugeno
utility functions (cf. Corollary~4.2 in \cite{CW3}).
\end{remark}

\section{An Example\label{sect example}}

We illustrate the results of the previous section with a simple example, where
preferences about travelling with four airlines $\mathrm{A}_{1},\mathrm{A}%
_{2},\mathrm{A}_{3},\mathrm{A}_{4}$ in economy class ($\mathrm{E}$) and first
class ($\mathrm{F}$) are modelled by pseudo-polynomial functions.
Table~\ref{table example}(a) shows a fictitious customer's evaluation of these
eight options, where $\mathrm{B}$,$\mathrm{N,G,V}$ stand for \textquotedblleft
bad\textquotedblright, \textquotedblleft neutral\textquotedblright,
\textquotedblleft good\textquotedblright\ and \textquotedblleft very
good\textquotedblright, respectively, and $\mathrm{D}$ means \textquotedblleft
don't know\textquotedblright. This table defines an overall preference
function $f\colon X_{1}\times X_{2}\rightarrow Y$, where%
\[
X_{1}:=\left\{  \mathrm{A}_{1},\mathrm{A}_{2},\mathrm{A}_{3},\mathrm{A}%
_{4}\right\}  ,~X_{2}:=\left\{  \mathrm{E},\mathrm{F}\right\}  ,~Y:=\left\{
\mathrm{B},\mathrm{N},\mathrm{G},\mathrm{V},\mathrm{D}\right\}  .
\]

\begin{table}[ptb]
\caption{The airline example}%
\label{table example}%
\centering\renewcommand{\arraystretch}{1.2}
\begin{tabular}
[t]{cl}%
\begin{tabular}
[t]{c}%
(a) The function $f\medskip$\\%
\begin{tabular}
[c]{cc|c}%
$x_{1}$ & $x_{2}$ & $f\left(  x_{1},x_{2}\right)  $\\\hline
$\mathrm{A}_{1}$ & $\mathrm{E}$ & $\mathrm{B}$\\
$\mathrm{A}_{1}$ & $\mathrm{F}$ & $\mathrm{B}$\\
$\mathrm{A}_{2}$ & $\mathrm{E}$ & $\mathrm{B}$\\
$\mathrm{A}_{2}$ & $\mathrm{F}$ & $\mathrm{D}$\\
$\mathrm{A}_{3}$ & $\mathrm{E}$ & $\mathrm{N}$\\
$\mathrm{A}_{3}$ & $\mathrm{F}$ & $\mathrm{G}$\\
$\mathrm{A}_{4}$ & $\mathrm{E}$ & $\mathrm{N}$\\
$\mathrm{A}_{4}$ & $\mathrm{F}$ & $\mathrm{V}$%
\end{tabular}
\end{tabular}
$\ \qquad$ &
\begin{tabular}
[t]{c}%
(b) The functions $\Phi_{k}^{-},\Phi_{k}^{+}\medskip$\\%
\begin{tabular}
[t]{p{0.75cm}|cc}%
$x_{1}$ & $\Phi_{1}^{-}\left(  x_{1}\right)  $ & $\Phi_{1}^{+}\left(
x_{1}\right)  $\\\hline
$\mathrm{A}_{1}$ & $\mathrm{B}$ & $\mathrm{B}$\\
$\mathrm{A}_{2}$ & $\mathrm{D}$ & $\mathrm{D}$\\
$\mathrm{A}_{3}$ & $\mathrm{G}$ & $\mathrm{G}$\\
$\mathrm{A}_{4}$ & $\mathrm{V}$ & $\mathrm{V}$%
\end{tabular}
\\
\\%
\begin{tabular}
[t]{p{0.75cm}|cc}%
$x_{2}$ & $\Phi_{2}^{-}\left(  x_{2}\right)  $ & $\Phi_{2}^{+}\left(
x_{2}\right)  $\\\hline
$\mathrm{E}$ & $\mathrm{B}$ & $\mathrm{N}$\\
$\mathrm{F}$ & $\mathrm{V}$ & $\mathrm{V}$%
\end{tabular}
\end{tabular}
\end{tabular}
\end{table}

It is plausible that $\mathrm{D}$ is better than $\mathrm{B,}$ worse than
$\mathrm{G}$, and incomparable with $\mathrm{N}$, hence the ordering of $Y$ is
the one given in Figure~\ref{fig lattice Y}. This is a distributive lattice
that can be embedded into the power set of a three-element set $U$ as shown in
Figure~\ref{fig lattice Y}. The figure does not indicate the representation of
each element of $Y$ as a subset of $U$, only the complements of the elements,
since this is all we need in order to perform the computations that follow.
(Note that $\mathrm{B}$ and $\mathrm{V}$ are the complements of each other.)%

\begin{figure}
[ptb]
\begin{center}
\includegraphics[
height=1.6708in,
width=3.3797in
]%
{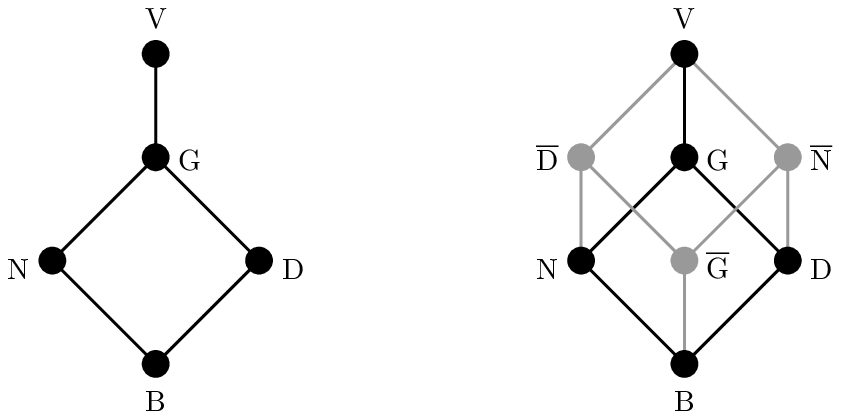}%
\caption{The lattice $Y$ and its embedding into a power set}%
\label{fig lattice Y}%
\end{center}
\end{figure}

For $y\in Y$ we have obviously $\operatorname{cl}\left(  y\right)
=\operatorname{int}\left(  y\right)  =y$, and for the three \textquotedblleft
extra\textquotedblright\ elements the closures and interiors can be read
easily from Figure~\ref{fig lattice Y}:%
\[%
\begin{tabular}
[c]{lll}%
\hphantom{$\operatorname{int}$}$\operatorname{cl}\left(  \overline{\mathrm{D}%
}\right)  =\mathrm{V},$ & \hphantom{$\operatorname{int}$}$\operatorname{cl}%
\left(  \overline{\mathrm{N}}\right)  =\mathrm{V},$ &
\hphantom{$\operatorname{int}$}$\operatorname{cl}\left(  \overline{\mathrm{G}%
}\right)  =\mathrm{V},$\\
\hphantom{$\operatorname{cl}$}$\operatorname{int}\left(  \overline{\mathrm{D}%
}\right)  =\mathrm{N},$ & \hphantom{$\operatorname{cl}$}$\operatorname{int}%
\left(  \overline{\mathrm{N}}\right)  =\mathrm{D},$ &
\hphantom{$\operatorname{cl}$}$\operatorname{int}\left(  \overline{\mathrm{G}%
}\right)  =\mathrm{B}.$%
\end{tabular}
\ \ \ \ \
\]

It is obvious that $0_{X_{2}}=\mathrm{E}$ and $1_{X_{2}}=\mathrm{F}$, but
$0_{X_{1}}$ and $1_{X_{1}}$ are not clear. However, from
Table~\ref{table example}(a) we can infer that $0_{X_{1}}=\mathrm{A}_{1}$ and
$1_{X_{1}}=\mathrm{A}_{4}$ if $f$ satisfies (\ref{eq BC}) at all. (If not,
then $f$ is not a pseudo-polynomial function.)

Table~\ref{table example}(b) shows the auxiliary functions $\Phi_{k}^{-}%
,\Phi_{k}^{+}$ corresponding to the function $f$. We give the details of the
computation of $\Phi_{2}^{+}\left(  \mathrm{E}\right)  $, the other values can
be calculated similarly:%
\begin{align*}
\Phi_{2}^{+}\left(  \mathrm{E}\right)   &  =%
{\displaystyle\bigwedge\limits_{x_{1}\in X_{1}}}
\operatorname{int}%
\bigl(%
f\left(  x_{1},\mathrm{E}\right)  \vee\overline{f(x_{1},\mathrm{F})}%
\bigr)%
\\
&  =\operatorname{int}(\mathrm{B}\vee\overline{\mathrm{B}})\wedge
\operatorname{int}(\mathrm{B}\vee\overline{\mathrm{D}})\wedge
\operatorname{int}(\mathrm{N}\vee\overline{\mathrm{G}})\wedge
\operatorname{int}(\mathrm{N}\vee\overline{\mathrm{V}})\\
&  =\operatorname{int}(\mathrm{V})\wedge\operatorname{int}(\overline
{\mathrm{D}})\wedge\operatorname{int}(\overline{\mathrm{D}})\wedge
\operatorname{int}(\mathrm{N})\\
&  =\mathrm{V}\wedge\mathrm{N}\wedge\mathrm{N}\wedge\mathrm{N}=\mathrm{N}.
\end{align*}

We can see that $\Phi_{k}^{-}\leq\Phi_{k}^{+}$ for $k=1,2$, therefore $f$ is a
pseudo-polynomial function by Theorem~\ref{thm characterization}.
Theorem~\ref{thm factorization--fi} implies that in any factorization
$f\left(  x_{1},x_{2}\right)  =p\left(  \varphi_{1}\left(  x_{1}\right)
,\varphi_{2}\left(  x_{2}\right)  \right)  $ of $f$, we must have $\varphi
_{1}=\Phi_{1}^{-}=\Phi_{1}^{+}$, while we have $2$ possibilities for
$\varphi_{2}$ (as $\varphi_{2}\left(  \mathrm{E}\right)  $ can be chosen to be
$\mathrm{B}$ or $\mathrm{N}$, and $\varphi_{2}\left(  \mathrm{F}\right)  $
must be $\mathrm{V}$). Thus there are $2$ pairs of functions $\left(
\varphi_{1},\varphi_{2}\right)  $, namely $\left(  \Phi_{1}^{-},\Phi_{2}%
^{-}\right)  $ and $\left(  \Phi_{1}^{+},\Phi_{2}^{+}\right)  $ that allow us
to factorize $f$. Theorem~\ref{thm factorization--fi} also shows that in both
cases one can use the polynomial function%
\begin{align*}
p_{0}\left(  y_{1},y_{2}\right)   &  =\mathrm{B}\vee\left(  y_{1}%
\wedge\mathrm{N}\right)  \vee\left(  y_{2}\wedge\mathrm{B}\right)  \vee\left(
\mathrm{V}\wedge y_{1}\wedge y_{2}\right) \\
&  =\left(  y_{1}\wedge\mathrm{N}\right)  \vee\left(  y_{1}\wedge
y_{2}\right)  =y_{1}\wedge\left(  y_{2}\vee\mathrm{N}\right)  .
\end{align*}

Computing the coefficients $c_{I}^{-},c_{I}^{+}$ for $\left(  \Phi_{1}%
^{-},\Phi_{2}^{-}\right)  $, one can see that in this case $p^{-}=p_{0}=p^{+}%
$, i.e., $p=p_{0}$ is the only polynomial function such that $f\left(
x_{1},x_{2}\right)  =p\left(  \Phi_{1}^{-}\left(  x_{1}\right)  ,\Phi_{2}%
^{-}\left(  x_{2}\right)  \right)  $. On the other hand, choosing $\left(
\varphi_{1},\varphi_{2}\right)  =\left(  \Phi_{1}^{+},\Phi_{2}^{+}\right)  $,
we obtain $p^{-}=y_{1}\wedge y_{2}$ and $p^{+}=p_{0}$, and these are the only
possibilities, since there is no polynomial function strictly between $p^{-}$
and $p^{+}$. Thus $f$ has altogether three factorizations:%
\begin{align*}
f\left(  x_{1},x_{2}\right)   &  =\Phi_{1}^{-}\left(  x_{1}\right)
\wedge\left(  \Phi_{2}^{-}\left(  x_{2}\right)  \vee\mathrm{N}\right) \\
&  =\Phi_{1}^{+}\left(  x_{1}\right)  \wedge\left(  \Phi_{2}^{+}\left(
x_{2}\right)  \vee\mathrm{N}\right) \\
&  =\Phi_{1}^{+}\left(  x_{1}\right)  \wedge\Phi_{2}^{+}\left(  x_{2}\right)
.
\end{align*}

Note that these factorizations are essentially the same, since $\Phi_{2}%
^{-}\vee\mathrm{N}=\Phi_{2}^{+}\vee\mathrm{N}=\Phi_{2}^{+}$. The meaning of
$\Phi_{2}^{+}$ is pretty obvious, and $\Phi_{1}^{+}$ shows the customer's
opinion about the four airlines, either based on past experience or on
information received from other sources (except for airline $\mathrm{A}_{2}$,
where the value $\mathrm{D}$ indicates the lack of information). The fact that
$\Phi_{1}^{+}\left(  x_{1}\right)  $ and $\Phi_{2}^{+}\left(  x_{2}\right)  $
are aggregated in a conjunctive manner indicates a pessimistic attitude:\ the
customer expects an enjoyable flight only if both the airline and the travel
class are good enough.

\section{Pseudo-polynomial Functions over Chains\label{sect chains}}

In this section we consider the case when $Y$ is a finite chain. As we will
see, in this case the results of Section~\ref{sect factorization} lead to a
generalization of Algorithm~SUFF presented in \cite{CW3}. As before, we will
suppose that $Y$ is a sublattice of $\mathcal{P}\left(  U\right)  $ for some
finite set $U$, with least element $\emptyset$ and greatest element $U$. We
may assume without loss of generality that $U=\left[  m\right]  =\left\{
1,2,\ldots,m\right\}  $, and $Y=\left\{  \left[  0\right]  ,\left[  1\right]
,\ldots,\left[  m\right]  \right\}  $, where $\left[  0\right]  =\emptyset$.
The closure of a set $S\subseteq U$ is the smallest set of the form $\left[
k\right]  $ that contains $S$, while the interior of $S$ is the largest set of
the form $\left[  k\right]  $ that is contained in $S$ (see
Figure~\ref{fig chains}). Formally, we have%
\[
\operatorname{cl}\left(  S\right)  =\left[  \max S\right]  ,\quad
\operatorname{int}\left(  S\right)  =\left[  \min\overline{S}-1\right]  .
\]
%

\begin{figure}
[ptb]
\begin{center}
\includegraphics[
height=1.9156in,
width=1.785in
]%
{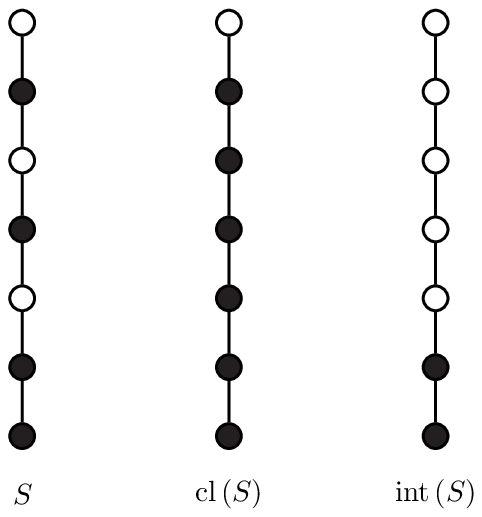}%
\caption{The closure and interior of a subset of a chain}%
\label{fig chains}%
\end{center}
\end{figure}

Let us assume that $f\colon\prod_{i\in\lbrack n]}X_{i}\rightarrow Y$ satisfies
(\ref{eq BC}). Then $f\left(  \mathbf{x}_{k}^{0}\right)  =\left[  u\right]
,f\left(  \mathbf{x}\right)  =\left[  v\right]  ,f\left(  \mathbf{x}_{k}%
^{1}\right)  =\left[  w\right]  $ with $u\leq v\leq w$, hence we have%
\begin{align*}
f\left(  \mathbf{x}\right)  \wedge\overline{f\left(  \mathbf{x}_{k}%
^{0}\right)  }  &  =\left\{  u+1,\ldots,v\right\}  ,\\
f\left(  \mathbf{x}\right)  \vee\overline{f\left(  \mathbf{x}_{k}^{1}\right)
}  &  =\left\{  1,\ldots,v,w+1,\ldots,m\right\}  .
\end{align*}
Therefore the terms in the definition of $\Phi_{k}^{-}$ and $\Phi_{k}^{+}$ can
be determined as follows:%
\begin{align}
\operatorname{cl}\bigl(f\left(  \mathbf{x}\right)  \wedge\overline{f\left(
\mathbf{x}_{k}^{0}\right)  }\bigr)  &  =\left\{
\begin{array}
[c]{ll}%
f\left(  \mathbf{x}\right)  , & \text{if }f\left(  \mathbf{x}_{k}^{0}\right)
<f\left(  \mathbf{x}\right)  ;\\
\emptyset, & \text{if }f\left(  \mathbf{x}_{k}^{0}\right)  =f\left(
\mathbf{x}\right)  ;
\end{array}
\right. \label{eq cl for chains}\\
\operatorname{int}\bigl(f\left(  \mathbf{x}\right)  \vee\overline{f\left(
\mathbf{x}_{k}^{1}\right)  }\bigr)  &  =\left\{
\begin{array}
[c]{ll}%
f\left(  \mathbf{x}\right)  , & \text{if }f\left(  \mathbf{x}_{k}^{1}\right)
>f\left(  \mathbf{x}\right)  ;\\
U, & \text{if }f\left(  \mathbf{x}_{k}^{1}\right)  =f\left(  \mathbf{x}%
\right)  .
\end{array}
\right.  \label{eq int for chains}%
\end{align}
Thus we obtain from Theorem~\ref{thm characterization} and
Remark~\ref{remark eq Fi_k-+ reformulated} the following characterization of
pseudo-polynomial functions valued in a chain.

\begin{theorem}
\label{thm characterization chains}If $Y$ is a finite chain, then a function
$f\colon\prod_{i\in\lbrack n]}X_{i}\rightarrow Y$ is a pseudo-polynomial
function if and only if it satisfies condition \textup{(\ref{eq BC})} and%
\[
f\left(  \mathbf{x}_{k}^{0}\right)  <f\left(  \mathbf{x}_{k}^{a_{k}}\right)
\text{ and }f\left(  \mathbf{y}_{k}^{a_{k}}\right)  <f\left(  \mathbf{y}%
_{k}^{1}\right)  \implies f\left(  \mathbf{x}_{k}^{a_{k}}\right)  \leq
f\left(  \mathbf{y}_{k}^{a_{k}}\right)
\]
holds for all $\mathbf{x},\mathbf{y}\in\prod_{i\in\lbrack n]}X_{i}$ and
$k\in\left[  n\right]  $, $a_{k}\in X_{k}$.
\end{theorem}

Let us now define the following three sets for any $k\in\left[  n\right]
,a_{k}\in X_{k}$, as in \cite{CW3}:%
\begin{align*}
\mathcal{W}_{a_{k}}^{f}  &  =\left\{  f\left(  \mathbf{x}\right)  \colon
x_{k}=a_{k}\text{ and }f\left(  \mathbf{x}_{k}^{0}\right)  <f\left(
\mathbf{x}\right)  <f\left(  \mathbf{x}_{k}^{1}\right)  \right\}  ,\\
\mathcal{L}_{a_{k}}^{f}  &  =\left\{  f\left(  \mathbf{x}\right)  \colon
x_{k}=a_{k}\text{ and }f\left(  \mathbf{x}_{k}^{0}\right)  <f\left(
\mathbf{x}\right)  =f\left(  \mathbf{x}_{k}^{1}\right)  \right\}  ,\\
\mathcal{U}_{a_{k}}^{f}  &  =\left\{  f\left(  \mathbf{x}\right)  \colon
x_{k}=a_{k}\text{ and }f\left(  \mathbf{x}_{k}^{0}\right)  =f\left(
\mathbf{x}\right)  <f\left(  \mathbf{x}_{k}^{1}\right)  \right\}  .
\end{align*}
From (\ref{eq cl for chains}) and (\ref{eq int for chains}) it follows that
$\Phi_{k}^{-}\left(  a_{k}\right)  =\bigvee\mathcal{L}_{a_{k}}^{f}\vee
\bigvee\mathcal{W}_{a_{k}}^{f}$ and $\Phi_{k}^{+}\left(  a_{k}\right)
=\bigwedge\mathcal{U}_{a_{k}}^{f}\wedge\bigwedge\mathcal{W}_{a_{k}}^{f}$,
hence the condition $\Phi_{k}^{-}\leq\varphi_{k}\leq\Phi_{k}^{+}$ in
Theorem~\ref{thm factorization--fi} can be reformulated as
follows:\renewcommand{\labelenumi}{\textup{(\alph{enumi})}} \renewcommand{\theenumi}{\labelenumi}

\begin{enumerate}
\item \label{item W}either $\mathcal{W}_{a_{k}}^{f}=\left\{  \varphi
_{k}\left(  a_{k}\right)  \right\}  $ or $\mathcal{W}_{a_{k}}^{f}=\emptyset$;

\item \label{item L}$\varphi_{k}\left(  a_{k}\right)  \geq\bigvee
\mathcal{L}_{a_{k}}^{f}$;

\item \label{item U}$\varphi_{k}\left(  a_{k}\right)  \leq\bigwedge
\mathcal{U}_{a_{k}}^{f}$.
\end{enumerate}

Thus by Theorem~\ref{thm factorization--fi}, $f$ is a pseudo-polynomial
function if and only if there are functions $\varphi_{k}$ satisfying the above
three conditions. If each $X_{k}$ is a bounded chain and $f$ is an
order-preserving function depending on all of its variables, then
\ref{item W},\ref{item L},\ref{item U} are equivalent to equation (4.5) in
\cite{CW3}, and Algorithm~SUFF does not return the value \textbf{false} if and
only if (\ref{eq Fi_k-+}) holds. Thus, in the finite case, Theorem~4.1 of
\cite{CW3} follows as a special case of Theorem~\ref{thm factorization--fi}.
Moreover, the results of Section~\ref{sect factorization} not only generalize
Algorithm~SUFF to arbitrary finite distributive lattices (instead of finite
chains) and to pseudo-polynomial functions (instead of Sugeno utility
functions), but they provide all possible factorizations of a given
pseudo-polynomial function $f$ (whereas Algorithm~SUFF\ constructs only one factorization).

\begin{remark}
\label{remark BC1 for chains}If the lattice $Y$ is a finite chain, as it is
the case in many applications, then any map $\varphi_{k}\colon X_{k}%
\rightarrow Y$ attains its minimum and its maximum at some points in $X_{k}$,
hence there exist elements $0_{X_{k}}$ and $1_{X_{k}}$ such that
(\ref{eq BC1}) holds. Thus, the boundary condition does not impose any
restriction on $\varphi_{k}$; the point is that we must know the elements
$0_{X_{k}}$ and $1_{X_{k}}$, even if we do not know the function $\varphi_{k}%
$. However, even this mild assumption can be released, by suitably modifying
the definition of the functions $\Phi_{k}^{-}$ and $\Phi_{k}^{+}$, so that
they do not refer to $0$ and $1$ anymore:%
\[
\Phi_{k}^{-}\left(  a_{k}\right)  :=\bigvee_{\substack{x_{k}=a_{k}%
\\\heartsuit\in X_{k}}}\operatorname{cl}\bigl(f\left(  \mathbf{x}\right)
\wedge\overline{f(\mathbf{x}_{k}^{\heartsuit})}\bigr)\text{,\quad}\Phi_{k}%
^{+}\left(  a_{k}\right)  :=\bigwedge_{\substack{x_{k}=a_{k}\\\heartsuit\in
X_{k}}}\operatorname{int}\bigl(f\left(  \mathbf{x}\right)  \vee\overline
{f(\mathbf{x}_{k}^{\heartsuit})}\bigr).
\]
To see that these new definitions yield the same functions as
(\ref{eq definition of Fi_k- and Fi_k+}), we only need to observe that we have
just added some new joinands to the join defining $\Phi_{k}^{-}$, but each of
these new joinands is dominated by some of the original joinands of
(\ref{eq definition of Fi_k- and Fi_k+}). Indeed, from (\ref{eq BC}) it
follows that $\overline{f(\mathbf{x}_{k}^{\heartsuit})}\leq\overline
{f(\mathbf{x}_{k}^{0})}$ for all $\heartsuit\in X_{k}$, hence%
\[
\operatorname{cl}\bigl(f\left(  \mathbf{x}\right)  \wedge\overline
{f(\mathbf{x}_{k}^{\heartsuit})}\bigr)\leq\operatorname{cl}\bigl(f\left(
\mathbf{x}\right)  \wedge\overline{f(\mathbf{x}_{k}^{0})}\bigr)\text{.}%
\]
Similarly, each of the new meetands that we have added to the meet defining
$\Phi_{k}^{+}$ is absorbed by some of the original meetands.

With these new definitions, we can apply Theorem~\ref{thm characterization} to
decide whether a given function $f$ is a pseudo-polynomial function, and we
can find all maps $\varphi_{k}$ that can appear in a factorization of $f$ with
the help of Theorem~\ref{thm factorization--fi}. Once we have the maps
$\varphi_{k}$, we can define $0_{X_{k}}$ and $1_{X_{k}}$ as%
\[
0_{X_{k}}:=\operatorname*{argmin}\limits_{x_{k}\in X_{k}}\varphi_{k}\left(
x_{k}\right)  \text{ and }1_{X_{k}}:=\operatorname*{argmax}\limits_{x_{k}\in
X_{k}}\varphi_{k}\left(  x_{k}\right)  ,
\]
and then we can use Theorem~\ref{thm factrization--fi and p} to find the
possible polynomial functions $p$. (Here $\operatorname*{argmin}$ and
$\operatorname*{argmax}$ denote the elements of $X_{k}$ where $\varphi_{k}$
attains its minimum and its maximum, respectively.) The price that we have to
pay for this generality is that the computation of $\Phi_{k}^{-}$ and
$\Phi_{k}^{+}$ is longer. Alternatively, we can apply \textquotedblleft
reverse engineering\textquotedblright\ to (\ref{eq BC}), as in
Section~\ref{sect example}, in order to find $0_{X_{k}}$ and $1_{X_{k}}$.
\end{remark}

\section{Concluding remarks\label{sect conclusion}}

We have extended the study of Sugeno utility functions over chains developed
in \cite{CW1,CW2,CW3} to the case of finite distributive lattices. We refined
the axiomatization given in \cite{CW1,CW3} by providing necessary and
sufficient conditions for a function defined on a Cartesian product of
arbitrary underlying sets and valued in a finite distributive lattice, to be
factorizable as a pseudo-polynomial function (\ref{eq:generalPol}). Moreover,
in doing so, we were able to furnish all possible factorizations, if such a
factorization exists, and we proposed a new procedure for constructing them,
which subsumes that of \cite{CW2,CW3} in the case when the codomain is a
finite chain.

Looking at directions for further research, we are inevitably drawn to the two
following topics. As mentioned, pseudo-polynomial functions play an important
role in multicriteria decision making since they subsume the so-called Sugeno
utility functions, which in turn are used to model preference relations: Say
that a preference relation $\preceq$ on $X_{1}\times\cdots\times X_{n}$ is
Sugeno representable if there is a Sugeno utility function $U\colon
X_{1}\times\cdots\times X_{n}\rightarrow Y$ such that $\mathbf{x}%
\preceq\mathbf{y}$ if and only if $U\left(  \mathbf{x}\right)  \leq U\left(
\mathbf{y}\right)  $. Given the results of the current paper, and following
the line of research developed in \cite{BouDubPraPir09,DPS,RGLC}, it is
natural to consider the following problem.

\begin{problem}
Axiomatize those preference relations that are Sugeno representable.
\end{problem}

This problem was solved in the realm of decision making under uncertainty in
\cite{DPS}.

The second emerging topic is of somewhat different nature. So far, we have
played within the setting where no information is missing. However, in
real-life situations this is rarely the case. Translating it into mathematical
terms, it is often the case that information about the functions we deal with
is incomplete. In other words, we are given partial functions in the sense
that they are not everywhere defined. Taking the simplest case where the only
aggregation functions considered are Sugeno integrals, we are faced with the
following interpolation problem.

\begin{problem}
Given a partial function $f\colon D\rightarrow X$, where $X$ is a distributive
lattice and $D\subseteq X^{n}$, give necessary and sufficient conditions for
the existence of a Sugeno integral $p\colon X^{n}\rightarrow X$ which
interpolates $f$ on all of its domain $D$, i.e., $p|_{D}=f$. If such a Sugeno
integral exists, provide a procedure to compute it.
\end{problem}

Again, this problem has been solved for the case of finite chains $X$ in
\cite{RGLC}.

These two problems constitute topics of current research being carried out by
the authors.

\subsubsection*{Acknowledgments.}

The first named author is supported by the internal research project
F1R-MTH-PUL-09MRDO of the University of Luxembourg. The second named author
acknowledges that the present project is supported by the
\hbox{T\'{A}MOP-4.2.1/B-09/1/KONV-2010-0005} program of National Development
Agency of Hungary, by the Hungarian National Foundation for Scientific
Research under grants no.\ K77409 and K83219, by the National Research Fund of
Luxembourg, and cofunded under the Marie Curie Actions of the European
Commission \hbox{(FP7-COFUND).}

\end{document}